\documentclass{amsart}
\usepackage{graphics}
\usepackage{graphicx}
\usepackage{epsfig}
\usepackage{amscd}
\usepackage{amssymb}
\usepackage{amsthm}
\usepackage{amsmath}
\numberwithin{equation}{section}
\usepackage{mathrsfs}
\usepackage{color}
\usepackage{hyperref}
\usepackage{bm}
\usepackage{float}
\usepackage{geometry}
\usepackage{tikz}
\usepackage{multirow}
\usepackage{verbatim}
\usepackage{geometry}
\usepackage{placeins}

\usetikzlibrary{matrix,arrows,arrows.meta,calc}
\usetikzlibrary{decorations.pathreplacing,decorations.markings}
\usetikzlibrary{shapes,positioning}
\tikzstyle arrowstyle=[scale=1]
\tikzstyle directed=[postaction={decorate,decoration={markings,
    mark=at position .5 with {\arrow[arrowstyle]{stealth}}}}]
\tikzstyle reverse directed=[postaction={decorate,decoration={markings,
    mark=at position .65 with {\arrowreversed[arrowstyle]{stealth};}}}]

\newtheorem{thm}{Theorem}[section]
\newtheorem{cor}[thm]{Corollary}
\newtheorem{prop}[thm]{Proposition}

\theoremstyle{plain}
\newtheorem{lem}[thm]{Lemma}

\theoremstyle{remark}

\theoremstyle{definition}
\newtheorem{defn}[thm]{Definition}
\newtheorem{ex}[thm]{Example}

\numberwithin{equation}{section}

\newcommand{\C}{\mathbb{C}}

\newcommand{\Z}{\mathbb{Z}}

\newcommand{\A}{\mathcal{A}}

\newcommand{\FF}{\mathscr{F}}
\newcommand{\E}{\mathscr{E}}

\newcommand{\OS}{\mathrm{OS}}

\newcommand{\sOS}{\mathscr{OS}}

\DeclareMathOperator{\rank}{rank}

\DeclareMathOperator{\colim}{colim}
\DeclareMathOperator{\coker}{coker}

\DeclareMathOperator{\Hilb}{Hilb}
\DeclareMathOperator{\NBC}{NBC}
\newcommand{\Mod}{\mathbf{Mod}}

\begin{document}

\title{Orlik--Solomon sheaf homology of geometric lattices}

\author{Ye Liu}
\address{Department of Pure Mathematics, Xi'an Jiaotong-Liverpool University, Suzhou, Jiangsu 215123, P. R. China}
\email{yeliumath@gmail.com}


\date{}



\begin{abstract}
We associate the Orlik--Solomon sheaf with a finite geometric lattice and compute its sheaf homology. We show that this homology concentrates in top degree, admitting a convolution-type decomposition into a principal ideal OS piece tensoring with a principal filter complement poset homology. Applications to uniform matroids provide interesting representations of symmetric groups.
\end{abstract}

\maketitle

\tableofcontents

\section{Introduction}

The Orlik--Solomon algebra is one of the fundamental algebraic invariants of a hyperplane arrangement.  If $\A$ is a complex hyperplane arrangement, Orlik and Solomon proved that the cohomology ring of the complement is determined by the intersection lattice $L(\A)$ and is naturally presented by generators indexed by the hyperplanes and relations indexed by dependent sets \cite{Orlik1980}.  This construction extends without reference to realizability: every geometric lattice, equivalently every simple matroid, has an associated Orlik--Solomon algebra.  Its graded pieces are free abelian groups with NBC bases, and the Hilbert series is the Poincar\'e polynomial of the lattice. Thus the Orlik--Solomon algebra connects two central themes of arrangement theory: the topology of complements and the combinatorics of geometric lattices.

Sheaf homology of posets provides another way to attach homological invariants to arrangements and matroids.  Everitt and Turner developed a theory of sheaf homology for graded atomic lattices and proved deletion--restriction long exact sequences for such homology theories.  They used this framework to compute the sheaf homology of the intersection lattice of a hyperplane arrangement with coefficients in the natural sheaf, recovering and extending a theorem of Lusztig \cite{Everitt2022}.  In subsequent work, they computed the corresponding homology with coefficients in exterior powers of the natural sheaf, using Boolean covers and cellular methods \cite{Everitt2022a}.  These results suggest that natural algebraic coefficient systems on geometric lattices often have computable and highly structured sheaf homology.

The purpose of this paper is to study the analogous coefficient system obtained from local Orlik--Solomon algebras.  Let $L$ be a geometric lattice of rank $\ell$ and let $x\in L$ be a flat.  The principal ideal $L_{\leq x}$ is again a geometric lattice, so it has its own Orlik--Solomon algebra $\OS^\bullet(L_{\leq x})$.  If $x\leq y$, then $L_{\leq x}\subseteq L_{\leq y}$.  Besides the natural inclusion of local Orlik--Solomon algebras, there is also a natural surjection
\[
\pi_{y,x}:\OS^\bullet(L_{\leq y})\longrightarrow \OS^\bullet(L_{\leq x})
\]
obtained by killing all atom generators not lying below $x$.  These maps define a sheaf of graded algebras $\sOS^\bullet$ on $L$.  We call it the Orlik--Solomon sheaf of $L$.

Our main result computes the sheaf homology of $P=L\setminus\{\hat0\}$ with coefficients in each graded piece $\sOS^p$. We prove that the sheaf homology concentrates in top degree $H_{\ell-1}(P;\sOS^p)$, which has a convolution-type decomposition with a factor the top OS group of a principal ideal and a factor the poset homology of a principal filter complement, which is a geometric semilattice.
We also describe bases for these homology groups.  The OS factor has its usual NBC basis.  For the semilattice factor, Ziegler's $\beta$-NBC bases index the basic cycles in the top homology.

Finally, we examine examples with symmetric group actions.  For Boolean lattices, the only positive-degree contribution occurs in top OS degree and the two sign representations cancel, giving the trivial representation.  Rank-two lattices give a complete decomposition into Specht modules.  For the uniform matroid $U_{r,n}$, the $p<r$ pieces are described by induction from $\mathfrak S_p\times\mathfrak S_{n-p}$ and Littlewood--Richardson coefficients, while the $p=r$ piece is the Kronecker square of a hook Specht module.  These examples show that the Orlik--Solomon sheaf homology naturally produces interesting representations of symmetric groups.

The paper is organized as follows.  Section 2 recalls the necessary background on geometric lattices, geometric semilattices, poset sheaf homology, Orlik--Solomon algebras, and defines the Orlik--Solomon sheaf. Section 3 proves the main homology computation, the Euler characteristic formula, and the basis construction.  Section 4 discusses examples, with emphasis on symmetric group representations arising from Boolean lattices, rank-two lattices, and uniform matroids.

\section{Preparations}

\subsection{Posets and geometric lattices}
We refer to \cite{Stanley2012} for a general introduction to posets and \cite{Wachs2007} for poset topology. All posets in what follows are assumed to be {\it finite}.

Let $P=(P,\leq)$ be a poset. For $x< y$ in $P$, we say $y$ {\it covers} $x$, written as $x\lessdot y$ or $y \gtrdot x$, if $x\leq z\leq y$ implies $z=x$ or $z=y$. For $x\in P$, define $P_{\leq x}=\{y\in P\mid y\leq x\}$. We also define $P_{\geq x},P_{<x}$ and $P_{>x}$ similarly. For $x\leq y$ in $P$, we define the interval $[x,y]=[x,y]_P=\{z\in P\mid x\leq z\leq y\}=P_{\geq x}\cap P_{\leq y}$. A subposet $Q$ of $P$ is called an {\it ideal} if $x\in Q,y\in P$ and $y\leq x$ implies $y\in Q$. An ideal is principal if it is of the form $P_{\leq x}$. Similarly, a subposet $Q$ of $P$ is called a {\it filter} if $x\in Q,y\in P$ and $x\leq y$ implies $y\in Q$. A filter is principal if it is of the form $P_{\geq x}$. 

For a poset $P$, we define the M\"obius function $\mu=\mu_P$ on intervals $[x,y]$ of $P$ recursively as
\[
\mu(x,y)=\begin{cases}
    1, & x=y,\\
    -\sum_{x\leq z<y}\mu(x,z), &x<y.
\end{cases}
\]

A chain of $P$ is a non-empty totally ordered subposet $\sigma=(x_0<x_1<\cdots<x_m)$. We say this chain has length $\ell(\sigma)=m$. All chains of $P$ form a simplicial complex, called the {\it order complex} $\Delta P$ of $P$. Note that the dimension of $\Delta P$ is the maximum of $\{\ell(\sigma)\mid \sigma\in\Delta P\}$. If every maximal chain of $P$ has the same length $\ell$, we say $P$ is {\it graded} (or {\it ranked}) of rank $\ell$. In this case, we define the rank function $\rank:P\to\{0,1,\ldots,\ell\}$ by $\rank(s)=0$ if $s\in P$ is minimal and $\rank(y)=\rank(x)+1$ whenever $x\lessdot y$. In a graded poset $P$, we denote
\[
P_k:=\{x\in P\mid \rank(x)=k\},
\]
and the truncation
\[
P_{\leq k}:=\{x\in P\mid \rank(x)\leq k\}, ~P_{\geq k}:=\{x\in P\mid \rank(x)\geq k\}.
\]

We say a poset $P$ has a bottom $\hat0$ if there is an element $\hat0\in P$ such that $\hat0\leq x$ for all $x\in P$. Similarly, we say $P$ has a top $\hat 1$ is there is an element $\hat 1$ such that $x\leq \hat1$ for all $x\in P$. There is a useful formula connecting M\"obius function to poset topology.
\begin{thm}[Philip Hall's theorem, \cite{Wachs2007} Proposition 1.2.6]\label{Hall}
    For a poset $P$ with $\hat{0}$ and $\hat1$,
    \[
    \mu(\hat0,\hat1)=\widetilde{\chi}(\Delta(P\setminus\{\hat0,\hat1\})).
    \]
\end{thm}

For a graded poset $P$ with a bottom $\hat0$, the Poincar\'e polynomial of $P$ is defined as
\[
\pi(P,t):=\sum_{x\in P}\mu(\hat0,x)(-t)^{\rank(x)}.
\]

A {\it lattice} $L$ is a poset in which the operations {\it join} and {\it meet} are defined for all pairs $x,y\in L$. Recall that the join $x\vee y$ and the meet $x\wedge y$ represent the least upper bound and the greatest lower bound for $x,y\in L$ respectively. If only meet (resp. join) is defined for all pairs, we shall call the poset a meet (resp. join) {\it semilattice}.
Note that in a lattice $L$, join and meet are associative so we write e.g. $\vee\{x,y,z\}=x\vee y\vee z$. A lattice $L$ has a bottom $\hat{0}$ and a top $\hat{1}$. 
We denote by $A=A(L)$ the set of {\it atoms} of $L$, that are elements covering $\hat{0}$. We call a lattice $L$ {\it atomic} if every element is a join of atoms, where we agree that an atom is the join of itself and $\hat{0}$ is the join of $0$ atom.

A ranked lattice $L$ is {\it semimodular} if the following {\it semimodular inequality} holds for all $x,y\in L$,
\[
\rank(x)+\rank(y)\geq \rank(x\wedge y)+\rank(x\vee y).
\]

\begin{defn}
    A {\it geometric lattice} $L$ is a ranked lattice, which is both atomic and semimodular. An element of a geometric lattice is called a \emph{flat}.
\end{defn}

\begin{ex}
    Let $\A$ be a central hyperplane arrangement, the intersection lattice $L(\A)$ is a geometric lattice (\cite{Stanley2012} Proposition 3.11.2). More generally, a lattice is geometric if and only if it is the flat lattice of a matroid (\cite{Oxley2011} Theorem 1.7.5). In fact, there is a bijection between geometric lattices and {\it simple} matroids (\cite{Oxley2011} p.54).
\end{ex}

There is also an affine analogue of geometric lattice. Let $S$ be a ranked meet semilattice. The atoms of $S$ are the rank $1$ elements. A set $I$ of atoms of $S$ is called {\it independent} if the join $\vee I$ is defined and $\rank(\vee I)=|I|$.

\begin{defn}
    A {\it geometric semilattice} is a ranked meet semilattice, which is atomic and the collection of independent sets of atoms form a matroid.
\end{defn}

There is an alternative description of geometric semilattices.
\begin{thm}[\cite{Wachs1986} Theorem 2.1]
    A ranked meet semilattice $S$ is geometric if and only if
    \begin{enumerate}
        \item every interval of $S$ is a geometric lattice, and
        \item whenever $\rank(x)<\rank(\vee I)=|I|$, where $x\in S$ and $I$ is an independent set of atoms of $S$, there is an atom $a\in I$ such that $a\not\leq x$ and $a\vee x$ exists.
    \end{enumerate}
\end{thm}

An important characterization of geometric semilattices is the following theorem.
\begin{thm}[\cite{Wachs1986} Theorem 3.2]\label{completion}
    A poset $S$ is a geometric semilattice if and only if there is a geometric lattice $L$ and an atom $a$ of $L$, such that
    \[
    S=L\setminus L_{\geq a}.
    \]
    Furthermore, $L$ and $a$ are uniquely determined by $S$. We call the pointed lattice $(L,a)$ the completion of $S$. 
\end{thm}
To produce a geometric semilattice, we have a more flexible way.
\begin{prop}[\cite{Wachs1986} Corollary 4.5]\label{filtercomplement}
    Let $L$ be a geometric lattice and $x\in L$. Then $L\setminus L_{\geq x}$ is a geometric semilattice.
\end{prop}

\subsection{Sheaf homology of posets}
We follow Everitt--Turner's approach \cite{Everitt2022}.
Let $R$ be a commutative ring with $1$ and $R\Mod$ the category of $R$-modules. A sheaf of $R$-modules on a poset $P$ is a contravariant functor $\FF:P\to R\Mod$, where $P$ is interpreted as a category in the usual way. For $x\leq y$ in $P$, we write $\FF^y_x: \FF(y)\to \FF(x)$ for the structure map.
Note that such $\FF$ is in fact a presheaf, but in our discrete setting, there is no essential difference between presheaves and sheaves.

For a sheaf $\FF$ on $P$, the colimit $\colim^P\FF$ is the quotient of $\bigoplus_{x\in P}\FF(x)$ by the submodule generated by $a_y-\FF^y_x(a_y)$ for $a_y\in \FF(y)$ and $x\leq y$. The higher colimit $\colim_*^P$ is defined as the left derived functor $L_*\colim^P$. The homology of $P$ with coefficients in the sheaf $\FF$ is the higher colimit evaluated at $\FF$
\[
H_*(P;\FF)=\colim_*^P\FF.
\]
The sheaf homology can be computed from a chain complex. For a sheaf $\FF$ on $P$, let
\[
T_n(P;\FF)=\bigoplus_{\sigma=(x_n<\cdots< x_0)\in\Delta P}\FF(x_0).
\]
Let $\sigma=(x_n< \cdots< x_0)\in \Delta P$ and $s\in \FF(x_0)$, write $s_{\sigma}$ for the element of $T_n(P;\FF)$ that has value $s$ in the component indexed by $\sigma$ and value $0$ in all other components. Then define the differential $d: T_n(P;\FF)\to T_{n-1}(P;\FF)$ by
\[
d(s_{\sigma})=\FF^{x_0}_{x_1}(s)_{d_0\sigma}+\sum_{i=1}^n(-1)^is_{d_i\sigma},
\]
where $d_i(x_n<\cdots< x_0)=(x_n<\cdots< \widehat{x_i}<\cdots< x_0)$ ($\widehat{x_i}$ means deleting $x_i$). 
\begin{prop}[\cite{Gabriel1967} Appendix II, see also \cite{Everitt2022} Section 2.2]\label{chaincomplex}
    We have an isomorphism
\[
H_*(P;\FF)\cong H(T_*(P;\FF)).
\]
\end{prop}

Some easy observations of poset sheaf homology are as follows.
\begin{lem}[\cite{Everitt2022} Lemma 1]\label{basiclemma}
    Let $P$ be a poset and $\FF$ a sheaf on $P$.
    \begin{enumerate}
        \item If $P$ is graded, then $H_i(P;\FF)\neq 0$ only if $0\leq i\leq \rank P$.
        \item If $P$ has a top or a bottom, and $\FF=\underline{M}$ is the constant sheaf of $R$-module $M$, then $H_0(P;\underline{M})=M$ and $H_i(P;\underline{M})=0$ for $i>0$.
        \item If $P$ has a bottom $\hat{0}$, then $H_0(P;\FF)=\FF(\hat{0})$ and $H_i(P;\FF)=0$ for $i>0$.
        \item If $P$ has a top $\hat{1}$, and $\FF(\hat{1})=0$, then $H_*(P;\FF)\cong H_*(P\setminus\{\hat{1}\};\FF)$.
    \end{enumerate}
\end{lem}

Let $P$ be a poset with a bottom $\hat{0}$ and $\FF$ a sheaf on $P$. We define an augmentation map 
\[
\epsilon: T_0(P\setminus\{\hat{0}\};\FF)\to\FF(\hat{0})
\]
as the sum of structure maps $\FF^x_{\hat{0}}$ over $x\in P\setminus\{\hat{0}\}$. Then we obtain the augmented complex $\widetilde{T}_*(P\setminus\{\hat{0}\};\FF)$, whose homology $\widetilde{H}_*(P\setminus\{\hat{0}\};\FF)$ is called the reduced homology of $P\setminus\{\hat{0}\}$ with coefficients in the sheaf $\FF$. The map $\epsilon$ induces $\epsilon_*: H_0(P\setminus\{\hat{0}\};\FF)\to \FF(\hat{0})$, which coincides with the map $\colim^{P\setminus\{\hat{0}\}}\FF\to\FF(\hat{0})$ induced by $\FF^x_{\hat{0}}$ using the universality of colimit. We have
\[
\widetilde{H}_i(P\setminus\{\hat{0}\};\FF)\cong \begin{cases}
    H_i(P\setminus\{\hat{0}\};\FF), & i>0\\
    \ker\epsilon_*, & i=0\\
    \coker\epsilon_*, & i=-1
\end{cases}
\]

\subsection{OS algebra and OS sheaf}

Let $L$ be a geometric lattice. Fix a total order of the atom set $A=L_1=\{a_1,\ldots,a_n\}$ and introduce a set of symbols $E=\{e_1,\ldots,e_n\}$ in one-to-one correspondence to $A$. 
Let $R$ be a ring and $\wedge^\bullet_RE$ be the exterior algebra over $R$ generated by degree $1$ generators $\{e_1,\ldots,e_n\}$. For an index set $I=\{i_1,\ldots,i_k\}\subseteq [n]=\{1,2,\ldots,n\}$ with $i_1<\cdots<i_k$, write $A_I=\{a_{i_1},\ldots,a_{i_k}\}$, $E_I=\{e_{i_1},\ldots,e_{i_k}\}$ and $e_I=e_{i_1}\cdots e_{i_k}\in \wedge^k_RE$.
The degree $k$ piece $\wedge_R^kE$ is a free $R$-module with basis $\{e_I\mid I\subseteq [n],|I|=k\}$. There is a boundary operator $\partial: \wedge^k_RE\to \wedge^{k-1}_RE$ defined by
\[
\partial(e_{i_1}\cdots e_{i_k})=\sum_{j=1}^k(-1)^{j-1}e_{i_1}\cdots \widehat{e_{i_j}}\cdots e_{i_k},
\]
where $\widehat{e_{i_j}}$ means deleting $e_{i_j}$. 
We say $A_I(I\subseteq [n])$ is a {\it circuit} if it is minimally dependent.
\begin{defn}
Let $L$ be a geometric lattice with atom set $A$. Define the Orlik--Solomon (OS) algebra of $L$ by
\[
\OS^\bullet_R(L):=\wedge^\bullet_RE/\langle \partial e_I\mid A_I \text{ is a circuit}\rangle.
\]
Note that the ideal in the quotient is graded and thus $\OS^\bullet_R(L)$ is graded.
\end{defn}

Orlik--Solomon \cite{Orlik1980} proved that when $L=L(\A)$ is the intersection lattice of a complex hyperplane arrangement $\A$, the OS algebra is isomorphic to the cohomology ring of the complement $M(\A)$. In particular, the Hilbert series of the OS algebra is equal to the Poincar\'e polynomial of the complement.

Let $L$ be a general geometric lattice. For a circuit $A_I=\{a_{i_1},\ldots,a_{i_k}\}(i_1<\cdots<i_k)$ of $L$, its broken circuit is $A_{I\setminus\{\min I\}}=\{a_{i_2},\ldots,a_{i_k}\}$. A subset $A_I$ is said to be NBC, or no-broken-circuit, if it contains no broken circuit. Let us denote by $\NBC(L)$ the set of NBC subsets of $A$, and $\NBC^p(L)=\{A_I\in\NBC(L)\mid |I|=p\}$.

\begin{prop}[\cite{Orlik1992} Chapter 3]\label{OSproperties}
Let $L$ be a geometric lattice.
\begin{enumerate}
    \item The graded piece $\OS^p_R(L)$ is a free $R$-module with basis $\{\overline{e_I}\mid A_I\in \NBC^p(L)\}$.
    \item The Hilbert series of $\OS^\bullet_R(L)$ is equal to the Poincar\'e polynomial of $L$: $\Hilb(\OS^\bullet_R(L),t)=\pi(L,t)$.
    \item The natural map $\bigoplus_{x\in L_p}\OS^p_R(L_{\leq x})\to\OS^p_R(L)$ is an isomorphism of $R$-modules.
\end{enumerate}

\end{prop}

Our goal is to define a sheaf on $L$ using OS algebras. Let $x\in L$ be a flat, then $L_{\leq x}$ is itself a geometric lattice with atom set $A_x:=A(L_{\leq x})=A(L)\cap L_{\leq x}$. Let us assign the OS algebra $\OS_R^\bullet(L_{\leq x})$ to $x$. We denote by $E_x\subseteq E$ the set of generators of $\OS_R^\bullet(L_{\leq x})$. When $x\leq y$ in $L$, we have $L_{\leq x}\subseteq L_{\leq y}$, $A_x\subseteq A_y$ and $E_x\subseteq E_y$. The inclusion $E_x\hookrightarrow E_y$ induces an inclusion of exterior algebras $\wedge^\bullet_RE_x\hookrightarrow\wedge^\bullet_RE_y$ preserving dependency, hence induces an inclusion of OS algebras $\iota_{x,y}:\OS_R^
\bullet(L_{\leq x})\hookrightarrow \OS_R^
\bullet(L_{\leq y})$. Conversely, the map $\pi_{y,x}:E_y\to E_x$
    \[
    \pi_{y,x}(e_i)=\begin{cases}
        e_i, & e_i\in E_x,\\
        0, & e_i\in E_y\setminus E_x,
    \end{cases}
    \]
induces a well-defined surjective graded algebra homomorphism
\[
    \pi_{y,x}:\OS_R^\bullet(L_{\leq y})\to \OS_R^\bullet(L_{\leq x}).
    \]
In fact, let $A_I$ be a circuit in $L_{\leq y}$. If $A_I\subseteq A_x$, then $\partial e_I\in \wedge^\bullet_RE_y$ is mapped to the same element in $\wedge^\bullet_RE_x$. If $A_I\not\subseteq A_x$, then $|A_I\setminus A_x|\geq 2$ since if $A_I\setminus A_x=\{a\}$, the condition $A_I$ is a circuit would force $a\in A_x$. Therefore every monomial occurring in $\partial e_I$ contains at least one generator $e_i\in E_y\setminus E_x$ and $\partial e_I$ is mapped to $0$. 

\begin{defn}
    Let $L$ be a geometric lattice. We define a sheaf $\sOS_R^{\bullet}$ of graded $R$-algebras on $L$ by
    \[
    \sOS_R^{\bullet}(x):=\OS_R^\bullet(L_{\leq x}), ~x\in L,
    \]
    and the structure map
    \[
    (\sOS_R^\bullet)^y_x=\pi_{y,x}: \sOS_R^\bullet(y)=\OS_R^\bullet(L_{\leq y})\to\OS_R^\bullet(L_{\leq x})=\sOS_R^\bullet(x), ~x,y\in L, x\leq y.
    \]
    We call $\sOS_R^\bullet$ the Orlik--Solomon sheaf on $L$. The degree $p$ part is the sheaf of $R$-modules
    \[
    \sOS_R^p(x)=\OS_R^p(L_{\leq x}),
    \]
    with structure maps restricted to degree $p$ part.
\end{defn}

\section{Results}
Let $L$ be a geometric lattice of rank $\ell\geq 2$ and $P:=L\setminus\{\hat{0}\}$. The main result of this paper is the computation of the sheaf homology of $P$ with coefficients in the OS sheaf. For ease of notation, we shall fix $R=\Z$ and omit the reference to $R$ in the subscript.

\subsection{Statement of the result}

We first have some immediate observations. 
If $p=0$, $\sOS^0$ is the constant sheaf $\underline{\Z}$, then $\widetilde{H}_*(P;\sOS^0)=0$ by Lemma \ref{basiclemma} (2). 
If $p>0$, $\sOS^p(\hat{0})=0$, the augmentation map $\epsilon: S_0(P;\sOS^p)\to \sOS^p(\hat{0})$ is trivial and hence $H_*(P;\sOS^p)=\widetilde{H}_*(P;\sOS^p)$. 
Therefore we shall compute $H_i(P;\sOS^p)$ for $p>0$.

\begin{thm}\label{main}
    Let $p>0$. Then
    \[
    H_i(P;\sOS^p)\cong \begin{cases}
        \bigoplus_{z\in L_p}\OS^p(L_{\leq z})\otimes_{\Z}\widetilde{H}_{\ell-2}(\Delta Q_z;\Z),~&i=\ell-1,\\
        0,& i\neq \ell-1,
    \end{cases}
    \]
    where $Q_z=\{x\in P\mid z\not\leq x\}$. Furthermore $\OS^p(L_{\leq z})$ is free abelian of rank $(-1)^{\rank z}\mu(\hat{0},z)$ and $\widetilde{H}_{\ell-2}(\Delta Q_z;\Z)$ is free abelian of rank $b_z=(-1)^{\ell-2}\sum_{x\in L_{\geq z}}\mu(\hat{0},x)$. Therefore $H_{\ell-1}(P;\sOS^p)$ is free abelian of rank $\sum_{z\in L_p}(-1)^{\rank z}\mu(\hat{0},z)b_z$.
\end{thm}

Define the graded OS sheaf Euler characteristic of $L$ by
\[
\chi^\OS(L,t):=\sum_{p}\chi^p(L)t^p,~\chi^p(L):=\sum_i(-1)^i\rank_{\Z}H_i(P;\sOS^p).
\]
For $x\in L$, recall the Poincar\'e polynomial of $L_{\leq x}$
\[
\pi_{L_{\leq x}}(t)=\sum_{z\in L_{\leq x}}\mu(\hat{0},z)(-t)^{\rank z}.
\]
\begin{cor}
    We have
    \[
    \chi^\OS(L,t)=1-\sum_{x\in L}\mu(\hat{0},x)\pi_{L_{\leq x}}(t).
    \]
\end{cor}
\begin{proof}
    By Theorem \ref{main},
    \[
    \chi^p(L)=\begin{cases}
        -\sum_{z\in L_p}(-1)^p\mu(\hat 0,z)\sum_{x\in L_{\geq z}}\mu(\hat0,x), & p>0;\\
        1, & p=0.
    \end{cases}
    \]
    Then
    \begin{align*}
        \chi^\OS(L,t)&=1-\sum_{p\geq 1}\left(\sum_{z\in L_p}(-1)^p\mu(\hat 0,z)\sum_{x\in L_{\geq z}}\mu(\hat0,x)\right)t^p\\
        &=1-\sum_{x\in L\setminus\{\hat0\}}\mu(\hat0,x)\left(\sum_{z\in (\hat0,x]}\mu(\hat0 ,z)(-t)^{\rank z}\right)\\
        &=1-\sum_{x\in L\setminus\{\hat0\}}\mu(\hat0,x)(\pi_{L_{\leq x}}(t)-1)\\
        &=1-\sum_{x\in L\setminus\{\hat0\}}\mu(\hat0,x)\pi_{L_{\leq x}}(t)+\sum_{x\in L\setminus\{\hat0\}}\mu(\hat0,x)
    \end{align*}
    Since $\sum_{x\in L}\mu(\hat0,x)=0,\mu(\hat 0,\hat0)=1$ and $\pi_{\hat0}(t)=1$, the desired formula follows.
\end{proof}

\subsection{Proof of the main theorem}
Let $z\in L_p$, define its top local OS group
\[
B_z:=\OS^p(L_{\leq z}).
\]
It is free abelian of rank $\rank_{\Z}B_z=(-1)^{p}\mu(\hat{0},z)$ (Proposition \ref{OSproperties}(1)(2)). With the total order of atoms, $B_z$ has the usual NBC basis indexed by the $p$-element NBC sets whose join is $z$. The Brieskorn decomposition (Proposition \ref{OSproperties}(3)) gives, for every $x\in L$,
\[
\sOS^p(x)\cong \bigoplus_{z\in L_p\cap L_{\leq x}}B_z.
\]
Under this decomposition, $\pi_{y,x}:\sOS^p(y)\to\sOS^p(x)$ is the projection onto those summands with $z\leq x$. This suggests a decomposition of the degree $p$ OS sheaf
\[
\sOS^p\cong \bigoplus_{z\in L_p} B_z\otimes_{\Z}\E_z,
\]
where $\E_z$ is the indicator sheaf of $L_{\geq z}$,
\[
\E_z(x)=\begin{cases}
    \Z, & z\leq x,\\
    0, & z\not\leq x,
\end{cases}
\]
with identity maps within $L_{\geq z}$ and zero maps otherwise. Then we have for $p>0$,
\[
H_i(P;\sOS^p)\cong \bigoplus_{z\in L_p} H_i(P;B_z\otimes_{\Z}\E_z)\cong \bigoplus_{z\in L_p}B_z\otimes_{\Z}H_i(P;\E_z).
\]
The computation is reduced to $H_i(P;\E_z)$. We use Proposition \ref{chaincomplex}.
\begin{lem}
    There is a natural chain isomorphism
    \[
    T_*(P;\E_z)\cong C_*(\Delta P,\Delta Q_z;\Z).
    \]
    Consequently,
    \[
    H_i(P;\E_z)\cong \widetilde{H}_{i-1}(\Delta Q_z;\Z).
    \]
\end{lem}
\begin{proof}
    For a chain $\sigma=(x_n<\cdots< x_0)$, the contribution to the complex $T_*$ is $\E_z(x_0)$, which is nontrivial ($\Z$) precisely when $x_0\in P_{\geq z}$. Since $P_{\geq z}$ is a principal filter, this is equivalent to the chain having at least one vertex in $P_{\geq z}$, or equivalently to $\sigma$ being a simplex of $\Delta P$ not lying in $\Delta Q_z$, where $Q_z=P\setminus P_{\geq z}$. Hence the chain groups are isomorphic. We also need to verify the boundary operators agree. If deleting the top vertex $x_0$ leaves a new top $x_1$ in $P_{\geq z}$, the sheaf map is the identity. If the new top $x_1$ is in $Q_z$, the sheaf map is zero, exactly as in the relative quotient. Since $P$ has maximum $\hat{1}$, $\Delta P$ is contractible. The long exact sequence then gives the result.
\end{proof}

To compute the homology of $\Delta Q_z$, we consider the poset $I_z:=Q_z\cup\{\hat{0}\}=L\setminus L_{\geq z}$. By Proposition \ref{filtercomplement}, $I_z$ is a geometric semilattice, and hence is CL-shellable by Corollary 7.3 of \cite{Wachs1986}. Moreover, we have the following purity result.
\begin{lem}
    $I_z$ is pure of rank $\ell-1$.
\end{lem}
\begin{proof}
    Take any $x\in I_z$. Since $z\not\leq x$ and $L$ is atomic, there is an atom $a\in A\cap L_{\leq z}$ such that $a\not\leq x$. We claim that there is a $y\in L_{\ell-1}$ satisfying
    \[
    x\leq y,~a\not\leq y.
    \]
    Hence $z\not\leq y$ because $a\leq z$. Therefore $y\in I_z$. This proves the purity: every $x\in I_z$ lies below a $y\in I_z$ of rank $\ell-1$.

    Next we have to prove the existence of such $y$. Choose $y$ maximal among elements satisfying
    \[
    x\leq y, ~a\not\leq y.
    \]
    Such $y$ exists because $x$ itself satisfies the condition. Suppose $\rank y<\ell-1$. Then the lattice $L_{\geq y}$ has rank at least $2$. Since $L$ is geometric, $L_{\geq y}$ is also geometric, hence atomic. Therefore $L_{\geq y}$ has at least two atoms, i.e. at least two covers of $y$. Because $a\not\leq y$, the join $y\vee a$ covers $y$ by semimodularity. Now choose another cover $y^\prime$ of $y$, $y^\prime\neq y\vee a$, then we have $a\not\leq y^\prime$. Indeed, if $a\leq y^\prime$, then $y\vee a\leq y^\prime$. But both $y\vee a$ and $y^\prime$ cover $y$, so this would force $y^\prime=y\vee a$, contradiction. Hence $y^\prime$ also satisfies
    \[
    x\leq y^\prime,~a\not\leq y^\prime,
    \]
    but $y<y^\prime$, contradicting the maximality of $y$. So $\rank y=\ell-1$.
\end{proof}

Therefore the order complex of $Q_z=I_z\setminus\{\hat{0}\}$ has the homotopy type of a wedge of $(\ell-2)$-spheres. Write the top betti number
\[
b_z:=\rank_{\Z}\widetilde{H}_{\ell-2}(\Delta Q_z;\Z).
\]

\begin{lem}
    We have a formula for the top betti number of $\Delta Q_z$:
    \[
    b_z=(-1)^{\ell-2}\sum_{x\in L_{\geq z}}\mu(\hat{0},x).
    \]
\end{lem}
\begin{proof}
    Adjoin a new maximum $\hat{1}_z$ to $I_z$. Then $Q_z=I_z\setminus\{\hat 1_z,\hat0\}$. We have the following,
    \[
    \widetilde{\chi}(\Delta Q_z)=\mu_{I_z\cup\{\hat{1}_z\}}(\hat{0},\hat{1}_z)=-\sum_{x\in I_z}\mu_{I_z\cup\{\hat{1}_z\}}(\hat{0},x)=-\sum_{x\in I_z}\mu_L(\hat{0},x)=\sum_{x\in L_{\geq z}}\mu_L(\hat{0},x),
    \]
    where the first equality is Theorem \ref{Hall}, the second is from the definition of M\"obius functions, the third is by the fact that M\"obius function $\mu_L(\hat0,\bullet)$ is invariant under restriction to ideals and the last is from $\sum_{x\in L}\mu(\hat0 ,x)=0$. Then the conclusion follows from the fact that $\Delta Q_z$ is a wedge of $(\ell-2)$-spheres.
\end{proof}

The proof of Theorem \ref{main} is completed.

\subsection{Basis construction}

Let us construct a basis for the free abelian group
\[
H_{\ell-1}(P;\sOS^p)\cong\bigoplus_{z\in L_p}\OS^p(L_{\leq z})\otimes_{\Z}\widetilde{H}_{\ell-2}(\Delta Q_z;\Z).
\]
Let $z\in L_p$, recall that
\[
\NBC^p(L_{\leq z})=\{A_I \mid |I|=p, A_I\text{ is NBC in }L_{\leq z}\}.
\]
Then by Proposition \ref{OSproperties}(1), $\{\overline{e_I}\mid A_I\in\NBC^p(L_{\leq z})\}$ is a basis for $\OS^p(L_{\leq z})$. 

The construction of a basis for $\widetilde{H}_{\ell-2}(\Delta Q_z;\Z)$ follows from the theory of shellable posets \cite{Bjorner1996}. Wachs--Walker \cite{Wachs1986} proved that geometric semilattices are CL-shellable. Ziegler \cite{Ziegler1992} proved that a geometric semilattice is EL-shellable by exhibiting an edge labeling. The corresponding shelling facets then give basic cycles for top homology. Let us briefly recall Ziegler's construction.

Recall that the atoms of $L$ are ordered $\{a_1,\ldots,a_n\}$ and $I_z=Q_z\cup\{\hat{0}\}=L\setminus L_{\geq z}$ is a pure geometric semilattice of rank $\ell-1$. Let $(\Lambda_z,\gamma_z)$ be the completion of $I_z$ (Theorem \ref{completion}), where $\Lambda_z$ is a geometric lattice and $\gamma_z$ is an atom of $\Lambda_z$ such that
\[
I_z\cong\{u\in\Lambda_z\mid \gamma_z\not\leq u\}.
\]
We order the atoms $A(\Lambda_z)$ by putting $\gamma_z$ first and using the fixed order on the old atoms from $I_z$. Let us identify $\widehat I_z=I_z\cup\{\hat 1_z\}$ with the subposet of $\Lambda_z$ consisting of the elements avoiding $\gamma_z$, together with the top element $\hat 1_z$ of $\Lambda_z$. Ziegler constructed an EL-shelling of $\hat I_z$ (\cite{Ziegler1992} Theorem 2.2). By Bj\"orner \cite{Bjorner1992} Proposition 7.6.4, the homology facets of this shelling of $\Delta Q_z$ correspond to the maximal chains of $\Delta Q_z$ with decreasing labels. These chains can be described using the $\beta$-no-broken-circuit bases of $\Lambda_z$.

Let us recall Ziegler's definition of $\beta$-no-broken-circuit basis. Let $B\subseteq A(\Lambda_z)$ be a basis of $\Lambda_z$, that is an independent set of atoms of cardinality $\ell=\rank \Lambda_z$. If $a\in A(\Lambda_z)\setminus B$, the basic circuit $C(B,a)$ is the unique circuit contained in $B\cup\{a\}$. Then $a$ is {\it externally active} with respect to $B$ if $a=\min C(B,a)$. If $b\in B$, define the basic cocircuit by
\[
C^*(B,b)=\{a\in A(\Lambda_z)\mid (B\setminus\{b\}\cup \{a\})\text{ is a basis of }\Lambda_z\}.
\]
Then $b$ is {\it internally active} if $b=\min C^*(B,b)$. Denote by $\mathrm{EA}(B)$ and $\mathrm{IA}(B)$ the set of externally active atoms and the set of internally active atoms respectively. Then we define
\[
\beta\NBC(\Lambda_z)=\{B:\text{ basis of }\Lambda_z\mid \mathrm{EA}(B)=\varnothing,\mathrm{IA}(B)=\{\gamma_z\}\}.
\]
A basis $B\in\beta\NBC(\Lambda_z)$ is called a $\beta$-no-broken-circuit basis.

\begin{thm}[\cite{Ziegler1992} Theorem 2.4]
    The maximal chains of $\hat I_z$ with decreasing EL labels are exactly the chains
    \[
    c_B: \hat 0 \lessdot b_1 \lessdot b_1\vee b_2 \lessdot \cdots \lessdot b_1\vee\cdots\vee b_{\ell-1}\lessdot \hat 1_z
    \]
    associated to $B=\{b_1,\ldots,b_{\ell-1},\gamma_z\}\in \beta\NBC(\Lambda_z)$ (listed in decreasing order with respect to the atom order of $\Lambda_z$).
\end{thm}

This shows that the top homology $\widetilde{H}_{\ell-2}(\Delta Q_z;\Z)$ has a basis indexed by $\beta\NBC(\Lambda_z)$. Ziegler (\cite{Ziegler1992} Corollary 2.7) explicitly constructed a basic cycle $\omega_B$ for $B\in\beta\NBC(\Lambda_z)$ and proved that $\{[\omega_B]\mid B\in\beta\NBC(\Lambda_z)\}$ is the basis of $\widetilde{H}_{\ell-2}(\Delta Q_z;\Z)$ associated to the shelling.

Therefore $H_{\ell-1}(P;\sOS^p)\cong\bigoplus_{z\in L_p}\OS^p(L_{\leq z})\otimes_{\Z}\widetilde{H}_{\ell-2}(\Delta Q_z;\Z)$ has a basis
\[
\{\overline{e_I}\otimes [\omega_B]\mid z\in L_p, A_I\in\NBC^p(L_{\leq z}), B\in\beta\NBC(\Lambda_z)\}.
\]

\section{Examples}
In this section, we take the coefficient $R=\C$ and consider examples whose automorphism group is the symmetric group. We describe the top homology $H_{\ell-1}(P;\sOS^p)$ as a representation of symmetric group $\mathfrak{S}_n$
\[
H_{\ell-1}(P;\sOS^p)\cong \bigoplus_{[z]\in L_p/\mathfrak{S}_n}\mathrm{Ind}^{\mathfrak{S}_n}_{(\mathfrak{S}_n)_z}\left(\OS^p(L_{\leq z})\otimes \widetilde{H}_{\ell-2}(\Delta Q_z) \right),
\]
where $(\mathfrak{S}_n)_z$ is the stabilizer of $z$ and $p>0$. Recall that 
\[
\dim \OS^p(L_{\leq z})=|\mu(\hat 0, z)|,~\dim \widetilde{H}_{\ell-2}(\Delta Q_z)=b_z=(-1)^{\ell-2}\sum_{x\in L_{\geq z}}\mu(\hat 0,x).
\]

In what follows, we denote by $S^{\lambda}$ the Specht module indexed by a partition $\lambda$. We refer to \cite{Stanley2024} for general facts of representation theory of symmetric groups.

\begin{ex}[Boolean lattice]
    Let $L=B_\ell$ be the Boolean lattice of rank $\ell\geq 2$. We identify the atoms with $\{1,2,\ldots,\ell\}$ and a rank $p$ flat with a subset of $\{1,2,\ldots,\ell\}$ of cardinality $p$. The automorphism group $\mathrm{Aut}(B_\ell)$ is isomorphic to the symmetric group $\mathfrak{S}_\ell$. In this case, the dimension of $\widetilde{H}_{\ell-2}(\Delta Q_z)$ is
    \[
    b_z=\begin{cases}
        1, & z=\hat{1};\\
        0, & z\neq\hat 1.
    \end{cases}
    \]
    The only interesting homology is $H_{\ell-1}(P;\sOS^\ell)\cong \OS^\ell(L)\otimes \widetilde{H}_{\ell-2}(\Delta(L\setminus\{\hat 0,\hat 1\}))$ with dimension $|\mu(\hat0,\hat1)|=1$. 
    Note that $\OS^\ell(L)=\C\cdot e_1\cdots e_\ell$ is the sign representation $S^{(1^\ell)}$ of $\mathfrak{S}_\ell$. The order complex of $B_\ell\setminus\{\hat0,\hat1\}$ is the barycentric subdivision of the boundary of an $(\ell-1)$-simplex. A direct analysis or an elaboration using Ziegler's cycle gives the cycle
    \[
    \omega=\sum_{\pi\in\mathfrak{S}_\ell}\mathrm{sgn}(\pi)[\{\pi(1)\}\lessdot\{\pi(1),\pi(2)\}\lessdot\cdots\lessdot\{\pi(1),\ldots,\pi(\ell-1)\}].
    \]
    Therefore $\tau\in\mathfrak{S}_\ell$ acts by $\tau[\omega]=\mathrm{sgn}(\tau)[\omega]$ and then $\widetilde{H}_{\ell-2}(\Delta(L\setminus\{\hat 0,\hat 1\}))\cong S^{(1^\ell)}$.
    
    As representations of $\mathfrak{S}_\ell$, we have
    \[
    H_{\ell-1}(P;\sOS^\ell)\cong S^{(1^\ell)}\otimes S^{(1^\ell)}\cong S^{(\ell)},
    \]
    the trivial representation.
\end{ex}

\begin{ex}[Rank $2$ lattices]
Let $L$ be a rank $2$ lattice with $n\geq 3$ atoms, that is the intersection lattice of a central arrangement of $n$ lines. The dimension formula gives
\[
\dim H_1(P;\sOS^1)=n(n-2),~\dim H_1(P;\sOS^2)=(n-1)^2.
\]
The automorphism group is $\mathfrak{S}_n$. First we compute the $\mathfrak{S}_n$-representation $H_1(P;\sOS^1)$. Choose an atom $z$ say $a_1$. Its stabilizer is $\mathfrak{S}_{n-1}$, permuting other atoms $\{a_2,\ldots,a_n\}$. Thus $\OS^1(L_{\leq a_1})$ is the trivial representation. On the other hand $Q_{a_1}=\{a_2,\ldots,a_n\}$ is a discrete set of $n-1$ atoms. Hence $\widetilde{H}_0(\Delta Q_{a_1})\cong S^{(n-2,1)}$ the standard representation. Therefore
\[
H_1(P;\sOS^1)\cong\mathrm{Ind}^{\mathfrak{S}_n}_{\mathfrak{S}_{n-1}}S^{(n-2,1)}\cong\begin{cases}
    S^{(n-1,1)}\oplus S^{(n-2,2)}\oplus S^{(n-2,1,1)}, & n\geq 4;\\
    S^{(2,1)}\oplus S^{(1,1,1)}, & n=3.
\end{cases}
\]
Next compute $H_1(P;\sOS^2)$. Let $z=\hat1$. Both $\OS^2(L)$ and $\widetilde{H}_0(\Delta Q_{\hat 1})$ are the standard representation $S^{(n-1,1)}$. Therefore
\[
H_1(P;\sOS^2)\cong S^{(n-1,1)}\otimes S^{(n-1,1)}\cong \begin{cases}
    S^{(n)}\oplus S^{(n-1,1)}\oplus S^{(n-2,2)}\oplus S^{(n-2,1,1)}, & n\geq 4;\\
    S^{(3)}\oplus S^{(2,1)}\oplus S^{(1,1,1)}, & n=3.
\end{cases}
\]
\end{ex}

\begin{ex}[Uniform matroid] Let $L=L(U_{r,n})$ be the flat lattice of the uniform matroid $U_{r,n}$ with $1<r\leq n$. So $L$ has rank $r$ and the truncation $L_{\leq r-1}$ is isomorphic to the truncation of Boolean lattice $(B_n)_{\leq r-1}$. The automorphism group is $\mathfrak{S}_n$. Note that this example covers the two previous examples.

Let us identify the atoms as $[n]=\{1,2,\ldots,n\}$. Fix $p<r$ and $q:=n-p$, a rank $p$ flat is a subset of $[n]$ of cardinality $p$. The action of $\mathfrak{S}_n$ on $L_p$ is transitive. Take $S=\{1,\ldots,p\}\in L_p$, then the stabilizer of $S$ is $\mathfrak{S}_p\times\mathfrak{S}_q$. The principal ideal $L_{\leq S}$ is Boolean of rank $p$, so
\[
\OS^p(L_{\leq S})\cong S^{(1^p)}\boxtimes S^{(q)}.
\]
The complex $\Delta Q_S$ is the barycentric subdivision of the simplicial complex
\[
K_S=\{X\subseteq [n]\mid |X|\leq r-1,S\not\subseteq X\}.
\]
Let $K$ be the full $(r-2)$-skeleton of the $(n-1)$-simplex $[n]$. Its top homology is $\widetilde{H}_{r-2}(K)\cong \wedge^{r-1}S^{(n-1,1)}$ as $\mathfrak{S}_n$-modules. Note that $K_S\subseteq K$ and the quotient pair $(K,K_S)$ records exactly the faces of $K$ that do contain $S$. Such a face has the form $S\sqcup Y$, where $Y\subseteq [n]\setminus S$ and $|Y|\leq r-p-1$. If we set $T:=[n]\setminus S$ with $|T|=q$, then the relative chain complex of $(K,K_S)$ is, up to a degree shift and an orientation factor from $S$, the augmented chain complex of the $(r-p-2)$-skeleton of the simplex on $T$. Therefore
\[
H_{r-2}(K,K_S)\cong S^{(1^p)}\boxtimes \bigwedge^{r-p-1}S^{(q-1,1)}
\]
as $\mathfrak{S}_p\times\mathfrak{S}_q$-modules. The homology short exact sequence of the pair $(K,K_S)$ gives a short exact sequence of $\mathfrak{S}_p\times\mathfrak{S}_q$-modules
\[
0\to \widetilde{H}_{r-2}(K_S)\to \mathrm{Res}^{\mathfrak{S}_n}_{\mathfrak{S}_p\times\mathfrak{S}_q}\bigwedge^{r-1}S^{(n-1,1)}\to S^{(1^p)}\boxtimes \bigwedge^{r-p-1}S^{(q-1,1)}\to 0
\]
that splits. Since $\mathrm{Res}^{\mathfrak{S}_n}_{\mathfrak{S}_p\times\mathfrak{S}_q}S^{(n-1,1)}\cong S^{(p-1,1)}\boxtimes S^{(q)}\oplus S^{(p)}\boxtimes S^{(q-1,1)}\oplus S^{(p)}\boxtimes S^{(q)}$, taking exterior powers gives the middle term
\[
\mathrm{Res}^{\mathfrak{S}_n}_{\mathfrak{S}_p\times\mathfrak{S}_q}\bigwedge^{r-1}S^{(n-1,1)}\cong \bigoplus_{a+b=r-1}\bigwedge^a S^{(p-1,1)}\boxtimes\bigwedge^b S^{(q-1,1)}\oplus \bigoplus_{a+b=r-2}\bigwedge^a S^{(p-1,1)}\boxtimes\bigwedge^b S^{(q-1,1)}.
\]
Since $\bigwedge^{p-1}S^{(p-1,1)}\cong S^{(1^p)}$, the quotient term
\[
S^{(1^p)}\boxtimes \bigwedge^{r-p-1}S^{(q-1,1)}\cong \bigwedge^{p-1}S^{(p-1,1)}\boxtimes \bigwedge^{r-p-1}S^{(q-1,1)}.
\]
Comparing the middle term with the quotient term, we obtain
\[
\widetilde{H}_{r-2}(\Delta Q_S)\cong \widetilde{H}_{r-2}(K_S)\cong \bigoplus_{a+b=r-1}\bigwedge^a S^{(p-1,1)}\boxtimes\bigwedge^b S^{(q-1,1)}\oplus \bigoplus_{\substack{a+b=r-2\\(a,b)\neq (p-1,r-p-1)}}\bigwedge^a S^{(p-1,1)}\boxtimes\bigwedge^b S^{(q-1,1)}.
\]
Recall that $\wedge ^k S^{(m-1,1)}=S^{(m-k,1^k)}$ for $1\leq k\leq m-1$ and $S^{(1^m)}\otimes S^{\lambda}\cong S^{\lambda^{\prime}}$ where $\lambda^{\prime}$ is the transpose of $\lambda$. We obtain
\begin{align*}
    H_{r-1}(P;\sOS^p)&\cong \mathrm{Ind}^{\mathfrak{S}_n}_{\mathfrak{S}_p\times\mathfrak{S}_q}\left(\OS^p(L_{\leq S})\otimes \widetilde{H}_{r-2}(\Delta Q_S)\right)\\
    &\cong\mathrm{Ind}^{\mathfrak{S}_n}_{\mathfrak{S}_p\times\mathfrak{S}_q}\left(\bigoplus_{a+b=r-1}S^{(a+1,1^{p-a-1})}\boxtimes S^{(q-b,1^b)}\oplus \bigoplus _{\substack{a+b=r-2\\(a,b)\neq (p-1,r-p-1)}}S^{(a+1,1^{p-a-1})}\boxtimes S^{(q-b,1^b)}\right)
\end{align*}
Using Littlewood--Richardson coefficients $c^\nu_{\lambda,\mu}$, we may express for a partition $\nu\vdash n$, the multiplicity of $S^\nu$ in $H_{r-1}(P;\sOS^p)$ is $\sum_{(a,b)\in I_{p,r,n}}c^\nu_{(a+1,1^{p-a-1}),(q-b,1^b)}$, where the index set
\begin{align*}
    I_{p,r,n}=&\{(a,b)\in \Z_{\geq 0}^2\mid a\leq p-1, b\leq n-p-1, a+b=r-1\}\\
    &\cup\{(a,b)\in\Z_{\geq 0}^2\mid a\leq p-1, b\leq n-p-1, a+b=r-2,(a,b)\neq (p-1,r-p-1)\}.
\end{align*}

Next we consider $p=r$ case. There is only one flat $\hat 1\in L_r$. Then
\[
\OS^r(L)\cong \widetilde{H}_{r-2}(\Delta Q_{\hat 1})\cong \bigwedge^{r-1}S^{(n-1,1)}\cong S^{(n-r+1,1^{r-1})}.
\]
Therefore
\[
H_{r-1}(P;\sOS^r)\cong S^{(n-r+1,1^{r-1})}\otimes S^{(n-r+1,1^{r-1})}\cong \bigoplus_{\nu\vdash n}(S^{\nu})^{\oplus g^{\nu}_{\lambda,\lambda}},
\]
where $\lambda=(n-r+1,1^{r-1})$ and $g^{\nu}_{\lambda,\lambda}$ is the Kronecker coefficient. Since $\lambda$ is of hook shape, a description of $g^{\nu}_{\lambda,\lambda}$ can be found in \cite{Remmel1989,Rosas2001}.
\end{ex}
\renewcommand\refname{References}
\bibliographystyle{hep}
\bibliography{OSsheaf}

\end{document}